\documentclass[11pt]{amsart}

\usepackage[margin=1in]{geometry}

\usepackage{url}

\usepackage{amsmath,amsfonts,amssymb,latexsym}
\usepackage[dvipsnames]{xcolor}
\usepackage{tikz}
\usetikzlibrary{shapes, arrows.meta, positioning}
\usepackage{verbatim}
\usepackage{esvect}
\usepackage{paralist}

\usepackage{cleveref}

\newcommand{\aS}{\mathfrak{S}}
\def\la{\lambda}
\def\al{\alpha}

\def\ov{\overline}
\def\rP{{\text {\rm P} } }
\newcommand{\sgn}{\textup{sign}}
\newcommand{\sort}{\textup{sort}}
\newcommand{\join}{\textup{\textsf{join}}}
\newcommand{\setpartition}{\Gamma}
\newcommand{\id}{\textup{id}}
\newcommand{\IZ}{\mathbb{Z}}
\newcommand{\ComCla}[1]{\textup{\textsf{#1}}}
\newcommand{\sharpP}{\ComCla{\#P}}
\newcommand{\PH}{\ComCla{PH}}

\newcommand{\CeqP}{\ComCla{C$_=$P}}
\newcommand{\coCeqP}{\ComCla{coC$_=$P}}
\newcommand{\Sigmap}{\ensuremath{\Sigma^{{\textup{p}}}}}
\renewcommand{\P}{\ComCla{P}}
\newcommand{\GapP}{\ComCla{GapP}}
\newcommand{\NP}{\ComCla{NP}}
\newcommand{\PP}{\ComCla{PP}}
\newcommand{\coNP}{\ComCla{coNP}}
\newcommand{\oP}{\ComCla{$\oplus$P}}

\newcommand{\sharpiDM}{\problem{\#$i$DM}}
\newcommand{\sharpIIDM}{\problem{\#2DM}}
\newcommand{\sharpIIIDM}{\problem{\#3DM}}
\newcommand{\sharpIVDM}{\problem{\#4DM}}
\newcommand{\mult}{\textup{mult}}

\newcommand{\problem}[1]{\mbox{{\normalfont\textsc{#1}}}} 

\renewcommand{\a}{\textup{\textbf{a}}}
\renewcommand{\b}{\textup{\textbf{b}}}
\renewcommand{\c}{\textup{\textbf{c}}}
\renewcommand{\d}{\textup{\textbf{d}}}

\DeclareMathOperator{\ind}{ind}

\def\nn{\mathbb N}
\def\IN{\mathbb N}
\def\cc{\mathbb C}
\DeclareMathOperator{\GL}{GL}
\def\cP{\mathcal P}

\numberwithin{equation}{section}
\newtheorem{conjecture}[equation]{Conjecture}

\newtheorem{theorem}[equation]{Theorem}

\newtheorem{proposition}[equation]{Proposition}

\newtheorem{lemma}[equation]{Lemma}

\newtheorem{corollary}[equation]{Corollary}
\newtheorem{ex}[equation]{Example}

\title[Positivity of the symmetric group characters is $\PH$-hard]{Positivity of the symmetric group characters  \\
is as hard as the polynomial time hierarchy}

\author[Christian~Ikenmeyer]{\ Christian~Ikenmeyer$^{\ddagger}$}

\author[Igor~Pak]{ \ Igor~Pak$^\star$}

\author[Greta~Panova]{ \ Greta~Panova$^{\dagger}$}

\date{\today}

\thanks{\thinspace ${\hspace{-.45ex}}^\ddagger$Department of Mathematics, 
University of Warwick, UK. 
\hskip.06cm Email: \hskip.06cm
\texttt{Christian.Ikenmeyer@warwick.ac.uk}}

\thanks{\thinspace ${\hspace{-.45ex}}^\star$Department of Mathematics,
UCLA, Los Angeles, CA~90095.
\hskip.06cm
Email:
\hskip.06cm
\texttt{pak@math.ucla.edu}}

\thanks{\thinspace ${\hspace{-.45ex}}^\dagger$Department of Mathematics,
 University of Southern California, Los Angeles, CA~90089.
\hskip.06cm Email: \hskip.06cm
\texttt{gpanova@usc.edu}}

\begin{document}

\begin{abstract} \small\baselineskip=9pt We prove that deciding the vanishing of the character of the symmetric group is $\CeqP$-complete. We use this hardness result to prove that the absolute value and also the square of the character are not contained in $\sharpP$, unless the polynomial hierarchy collapses to the second level.  This rules out the existence of any (unsigned) combinatorial description for the square of the characters.
As a byproduct of our proof we conclude that deciding positivity of the character is $\PP$-complete under many-one reductions, and hence $\PH$-hard under Turing-reductions.
\end{abstract}

\maketitle

\section{Introduction}
\subsection{Motivation}
\label{subsec:motivation}
Consider the following two classical identities from the representation theory of the symmetric group\footnote{The characters are the traces of the representation matrices, see e.g.~\cite[Ch.~2]{JK}.}:
\begin{equation}
\label{eq:RSK}
n! \ = \ \sum_{\la\vdash n} \, \big(\chi^\la(1)\big)^2 \ , \ \ \text{and}
\end{equation}
\begin{equation}
\label{eq:charsum}
n! \ = \ \sum_{\pi\in\aS_n} \big(\chi^\la(\pi)\big)^2 \quad\text{for all} \ \ \la\vdash n.
\end{equation}
Here $\chi^\la$ is the irreducible character of the symmetric group $\aS_n$ of the representation indexed by~$\la$,
and $\chi^\la(\pi)\in \IZ$ is its evaluation.  Both identities arise in a similar manner,
as squared norms of row and column vectors in the character table
of $\aS_n$, see~$\S$\ref{ss:finrem-White} for the context and generalizations.

Equalities such as these, are an invitation
for combinatorialists to search for natural bijections between the sets of
combinatorial objects counting both sides.  In both cases, the LHS is the
set $\aS_n$ of permutations of $n$ symbols.  For \eqref{eq:RSK}, the
RHS is the set of pairs of \emph{standard Young tableaux} of the same
shape with $n$ boxes.
The bijection between the set of permutations and the set of pairs of Young tableaux is
the celebrated \emph{Robinson--Schensted correspondence}, which is
fundamental in Algebraic Combinatorics, see \cite[Ch.~3]{Sag}
and \cite[$\S\S$7.11-14]{EC2}.  This correspondence has numerous
generalizations and is studied widely across many areas of mathematics
and applications, see e.g.\ \cite{And,BuS,DNV,KP,OC,OW}.

Similarly, for \eqref{eq:charsum}, one would want to give a bijection between $\aS_n$
and  a set of $n!$ many combinatorial objects that are partitioned naturally
into subsets of sizes $\big(\chi^\la(\pi)\big)^2$.
In this paper we prove that this approach would fail for the fundamental reason that
the RHS of~\eqref{eq:charsum} does not admit such an interpretation.
As the following theorem implies, it is unlikely that there exist
``sets of $\big(\chi^\la(\pi)\big)^2$ many combinatorial objects'' (see more on this below).

\begin{theorem}
\label{thm:main}
Let  $\chi^2 : (\la,\pi) \mapsto \big(\chi^\la(\pi)\big)^2$, where $\la\vdash n$ and $\pi \in \aS_n$.
If the function $\chi^2$ is contained in the complexity class $\sharpP$, then
$\coNP=\CeqP$.  Consequently\footnote{Indeed, Tarui (\cite{Tar91}, see also \cite{Gre93}) proves that $\ComCla{PH} \subseteq \NP^{\CeqP}$. Therefore, if $\coNP=\CeqP$, then $\PH \subseteq \NP^\CeqP = \NP^\coNP = \Sigmap_2$, and hence $\Sigmap_2=\PH$.}, if  $\chi^2 \in \sharpP$, then
the polynomial hierarchy collapses to the second level: $\PH = \Sigmap_2$.
\end{theorem}
The assumption $\PH \neq \Sigmap_2$ in the theorem is a widely believed standard complexity theoretic assumption.
From a combinatorial perspective, \Cref{thm:main} is much stronger than just saying
that the character squares are hard to compute. The theorem rules out that there
exists any positive combinatorial interpretation for the character squares,
even if ``positive combinatorial interpretation'' is interpreted in the widest
 possible sense.
Large parts of Enumerative and Algebraic Combinatorics deal with finding
explicit (positive) combinatorial interpretations of quantities,
while impossibility results such as \Cref{thm:main} are extremely rare,
see \S\ref{subsec:relwork}.

Note also how close the upper and lower bounds are.  Recall that the character square is
in $\GapP=\sharpP-\sharpP$, is always nonnegative, and yet is not in~$\sharpP$ by
the theorem unless the polynomial hierarchy collapses.  Our proof goes via showing that deciding the vanishing of $\chi^\la(\pi)$
is $\CeqP$-complete:

\begin{theorem}
\label{thm:CeqPhardnessofchar}
The language $\{(\la,\pi) \mid \chi^\la(\pi)=0\}$ is
$\CeqP$-complete under many-one reductions.
\end{theorem}

\Cref{thm:main} then follows from \Cref{thm:CeqPhardnessofchar} and \Cref{pro:collapse}.
The result in the title is a direct consequence of the reduction in
the proof of   Theorem~\ref{thm:CeqPhardnessofchar}.

\begin{theorem}
\label{thm:PH-char}
The language $L=\{(\la,\pi) \mid \chi^\la(\pi) \ge 0\}$ is $\PP$-complete under many-one reductions.
Consequently, $L$ is $\PH$-hard under Turing-reductions.
\end{theorem}

Indeed, since $\PH \subseteq \P^\PP$ by \cite{Tod89,Tod91},
it immediately follows that $L$ is $\PH$-hard under Turing reductions:
$\PH \subseteq \P^\PP \, \stackrel{\textup{Thm.~\ref{thm:PH-char}}}{\, \subseteq} \P^{\P^L} = \P^L$.
This derives the second part of the theorem from the first part.
As a side result we prove that computing the character is strongly $\GapP$-complete, see Theorem~\ref{thm:strongGapPhard}.

\subsection{$\sharpP$, $\GapP$ and combinatorial interpretations}
\label{subsec:sharpPGapP}
Let $\{0,1\}^*$ denote the set of finite length sequences of zeros and ones.
The length $|w|$ of a bit string $w\in\{0,1\}^*$ is defined as the number of its symbols.
For a set $S$ let $2^S$ be its power set, i.e., the set of all subsets of~$S$.

The class $\sharpP$ is commonly defined via nondeterministic Turing machines as follows, but we discuss another definition below.
For a nondeterministic Turing machine $M$ and a word $w \in \{0,1\}^*$ let $\textup{acc}_M(w)$ denote the number of accepting computation paths of $M$ on input~$w$.
The complexity class $\sharpP$ is defined as the class of those functions $f:\{0,1\}^*\to\IN$ for which a nondeterministic Turing machine $M$ exists with $\forall w \in \{0,1\}^*: f(w)=\textup{acc}_M(w)$.

There is no generally accepted definition of a \emph{positive combinatorial interpretation} of a function $f:\{0,1\}^*\to\IN$,
but intuitively the existence of a positive combinatorial interpretation of $f$ should mean the existence of a map $\varphi$ that assigns to each $w\in\{0,1\}^*$ a set $\varphi(w)$ whose cardinality is $f(w)$. For example, as in \eqref{eq:RSK}, $f(\la)=(\chi^\la(1))^2$, 
$$\varphi(\la)=\{(T,S)\mid T,S \textup{ semistandard tableaux of shape $\la$ with $|\la|$ boxes}\}.$$
But some additional properties on $\varphi$ are required to get a useful definition. Otherwise one could always define the trivial
\begin{equation}
\label{eq:trivial}
\varphi(w)=\{1,2,3,\ldots,f(w)\}.
\end{equation}
A meaningful definition that covers \eqref{eq:RSK} and also covers well-known cases from algebraic combinatorics such as for example the Littlewood-Richardson coefficient is the class $\sharpP$. In $\sharpP$ we impose two restrictions on $\varphi$. First, the description length of the elements in $\varphi(w)$ must be bounded by a polynomial in $|w|$, and second, there must be a polynomial time algorithm $V$ that on input $(w,x)$ decides whether or not $x \in \varphi(w)$. More precisely,
$f \in\sharpP$ if and only if there exists a univariate polynomial $t$ and
a polynomial time algorithm $V:\{0,1\}^*\times \{0,1\}^*\to\{0,1\}$ such that
\[
\forall w \in \{0,1\}^* : \quad f(w) = \sum_{\substack{x\in \{0,1\}^*\\|x|\leq t(|w|)}} V(w,x).
\]
The restriction of the computational resources of $V$ eliminates the problem in \eqref{eq:trivial}, at least for all $f$ that are hard to compute (which is the case in this paper).
Note that "polynomial computation time" is a robust notion that is independent of any formalization of (non-quantum) computation, and also works for Turing machines, which
Turing compares to a human ``computer'' in \cite[\S9]{Tur36}.
This point is excellently explained in \cite[Ch.~1]{AB09}.
A restriction on the description length of the elements of $\varphi(w)$ is also important, because otherwise for every computable function $f$ we can choose $\varphi(w)=\{(T_w,i) \mid 1 \leq i \leq f(w)\}$, where $T$ is a transcript of the performed computation steps when computing $f(w)$, which is an issue very similar to \eqref{eq:trivial}.
Slightly weaker version of both issues persist if ``unbounded'' is replaced by ``exponentially bounded''.
One could make the restrictions stronger than ``polynomially bounded'' to get subclasses of $\sharpP$, but our results already show non-containment in $\sharpP$.
Considering a growth behavior between polynomial and exponential is also possible, but it is not so clear how natural those classes are.

For example, the famous \emph{Littlewood--Richardson rule} states that the Littlewood--Richardson (LR) coefficient $c_{\la,\mu}^\nu$ equals the number of LR--tableaux of skew shape $\nu/\la$ and content $\mu$, hence the map $(\la,\mu,\nu)\mapsto c_{\la,\mu}^\nu$ is in $\sharpP$.
Here we already see an interesting subtlety: This argument works if the partitions are given as their Young diagrams, i.e., the partitions are given in unary, because otherwise writing down a single LR-tableau would require exponential space. The LR-coefficient is in $\sharpP$ for binary inputs, see e.g.\ \cite{Nar}, which follows from their interpretation as the number of integer points in a certain polytope, and not the LR-tableaux.
From the perspective of combinatorics, a ``combinatorial interpretation'' of the Littlewood--Richardson coefficient already follows from the former result. \Cref{thm:main} works in unary and hence also in binary.

Let us also remark that $\sharpP$ is the class of positive combinatorial interpretations if ``positive combinatorial interpretation'' is used in a very broad and all-encompassing sense. This means that a proof of the non-membership in $\sharpP$ such as \Cref{thm:main} is a very strong impossibility result, as it rules out also very complicated tableau constructions, including, e.g.,
those in \cite{Bla17,TY}.


The complexity class $\GapP := \sharpP-\sharpP$ is defined as the class of differences of
two $\sharpP$ functions, i.e., $\GapP = \{f-g \mid f,g\in\sharpP\}$.
Let $\GapP_{\geq 0}$ denote the subset of nonnegative functions in $\GapP$.
Many interesting functions in algebraic combinatorics are known to be in $\GapP_{\geq 0}$,
but conjectured to be in $\sharpP$.  See \cite{IP22,Pak19} for many such functions
arising from combinatorial inequalities.
The most famous  $\GapP_{\geq 0}$ functions are the subject of
of Stanley's survey \cite{Sta00} on positivity problems in algebraic combinatorics,
where he asked for positive combinatorial interpretations of the plethysm, Kronecker,
and Schubert coefficients. All these problems remain unresolved
(cf.~$\S$\ref{ss:finrem-Kron}).

Closer to the subject of this paper, Stanley considered
rows and column sums of the character table of $\aS_n$:
\begin{equation}\label{eq:row-column-sums}
a_\la  :=  \sum_{\mu\vdash n}  \chi^\la(\mu) \qquad \text{and}
\qquad b_\la  := \sum_{\mu\vdash n}  \chi^\mu(\la),
\end{equation}
respectively, see Problem~12 in~\cite{Sta00} for references for the nonnegativity of $a_\la$ and $b_\la$. Here $\chi^\la(\mu)$ denotes
the character value on permutations of cycle type~$\mu$.   Viewed as functions
with unary input, it is easy to see that $a_\la$ and $b_\la$ are in $\GapP_{\geq 0}$.
Stanley notes that
$b_\la = \bigl|\{\omega\in \aS_n \mid \omega^2=\sigma\}\bigr|$, where
$\sigma$ has cycle type~$\la$, which implies that $b_\la$ is in $\sharpP$.
Stanley asked for a positive combinatorial interpretation of~$a_\la$, which
remains an open problem (cf.~$\S$\ref{ss:finrem-Kron}).
\Cref{thm:main} could be seen as a critical reminder that there is the possibility
that the desired combinatorial interpretations might not exist.

\subsection{Related work}
\label{subsec:relwork}
The amount of work on characters of the symmetric groups is much too large to
be reviewed here, but let us note that they prominently appear in other
fields, see e.g.\ \cite{Dia,Pau,Ste}, and  have remarkable applications, see e.g.\  
\cite{EFP,MRS}.  On the other hand, the asymptotic proportion of
zeros in the character table remains open, see complementary discussions of the same
data in \cite[$\S$1.2]{Mil} and \cite[$\S$8.5]{PPV}.

Hepler \cite{Hep} proved that the computation of $\chi^\la(\pi)$ is $\sharpP$-hard under many-one reductions.
He does not study the vanishing problem of $\chi^\la(\pi)$.
The vanishing of the character $\chi^\la(\mu)$ was proved to be $\NP$-hard in \cite{PP_comp}.
It is noteworthy that the result in \cite{PP_comp} only holds for the problem where the input $(\la,\mu)$
is encoded in binary, i.e., instead of $\pi$ the second parameter is just the cycle type $\mu$ in binary.
Our results do not have such a restriction.
The relativizing closure properties of $\sharpP$ have been characterized in \cite{HVW95},
which can be generalized to prove non-containment in $\sharpP$ w.r.t.\ an oracle in several settings,
see \cite{IP22}.

In the combinatorics literature, the notion of a ``positive combinatorial
interpretation'' is used informally; these are also called \emph{manifestly positive}
combinatorial formulas, rules, expressions, etc.  This is to emphasize the
importance of positivity, as opposed to \emph{signed combinatorial formulas},
which typically refers to formulas in (subsets of) $\GapP$.  A complexity theoretic approach
in this setting was introduced in \cite{Wilf} (see also \cite{Pak18}).

For characters $\chi^\lambda(\pi)$, the $\GapP$ formula is famously given by the
\emph{Murnaghan--Nakayama rule} as the difference is the number of certain
\emph{rim hook tableaux}, see e.g.~\cite[$\S$4.10]{Sag} and \cite[$\S$7.17]{EC2}.
In this context, \cite[Cor.~7.5]{Sta84} gave a simple sufficient condition for
the vanishing $\chi^\la(\mu)=0$.

For Kronecker coefficients $g(\la,\mu,\nu)$, the $\GapP$ formula is given in
in~\cite{BI-kron} (see also \cite{CDW,PP_comp}).
For $\GapP$ formulas of plethysm and Schubert coefficients, see~\cite{FI20} and
\cite[Prop.\ 17.3]{PS}, respectively. In the context of
\emph{Geometric Complexity Theory} (GCT), the importance of
being in $\sharpP$ of plethysm and Kronecker coefficients was discussed in~\cite{Mul09}.
Kahle and Micha{\l}ek \cite{KM18} prove that plethysm coefficients are not
counting integer points in polytopes; this is a restricted notion compared to $\sharpP$
of interest both in Algebraic Combinatorics and GCT.

Stanton and White \cite{SW} gave a generalization of the Robinson--Schensted correspondence
for rim hook tableaux. This was used in \cite{White1,White2}
to obtain combinatorial proofs of two character identities:
first, of a generalization of~\eqref{eq:RSK} given in~\eqref{eq:charsumz},
and then of \eqref{eq:charsum},  but both proofs use an explicit
involution to cancel the signs.

Finally, the complexity classes that we study in this paper
are all standard and have been studied in numerous papers.
In particular, it is known that $\CeqP = \ComCla{coNQP}$ \cite{FGHP99}.

\medskip

\section{Preliminaries}
\subsection{Notation}
A subset $L\subseteq\{0,1\}^*$ is called a \emph{language}. We write $\overline{L} := \{0,1\}^*\setminus L$ to denote the complement of~$L$.
We write $\binom{S}{k}$ for the set of cardinality~$k$ subsets of~$S$.

We use $\nn=\{0,1,2,\ldots\}$ and $[n]=\{1,\ldots,n\}$. We denote by $\mathbb{Z}_k=\{1,\ldots,k\}$ the set of integers modulo $k$.
Let $\aS_n$ denote the group of permutations of $[n]$.

For a list of nonnegative integers $\a$ we call
$\ell(\la)=\max\{i \mid a_i > 0\}$ the \emph{length} of~$\a$.
A \emph{(weak) composition} of $n$ is sequence of nonnegative integers whose entries sum up to $n$.
An \emph{integer partition} $\la$ of $n$, denoted $\la \vdash n$, is a sequence of weakly decreasing nonnegative integers $(\la_1,\la_2,\ldots)$ which sum up to $n$. We write $|\la| = \sum_i \la_i$.

We treat compositions and partitions as vectors with componentwise addition and with the simultaneous rescaling of all components. We also allow adding vectors of different lengths by implicitly appending zeros to the shorter vector.
We write $\sort(\a)$ for the tuple that has the same entries as $\a$, but they are permuted so that they appear in weakly decreasing order.
We denote by $a^b$ the sequence $(a,a,\ldots,a)$ with $a$ appearing $b$ times. We write $\a=(a_1,\ldots,a_l)$ and $\b=(b_1,\ldots,b_l)$ for compositions and $|\a|=a_1+\ldots+a_l$ for their sum.
We use the nested list notation to represent concatenation of lists, e.g., $(\a,\b) = (a_1,a_2,\ldots,a_{\ell(\a)},b_1,b_2,\ldots,b_{\ell(\b)})$, and $(1,5,\d) = (1,5,d_1,d_2,\ldots,d_{\ell(\d)})$.

\subsection{Representation Theory}\label{ss:rep_theory}
Let $\chi^\la\in \cc[\aS_n]$ be the complex irreducible character of $\aS_n$
corresponding to partition $\la\vdash n$, i.e., for $\pi\in\aS_n$ we have that $\chi^\la(\pi)$ equals the trace of the representation matrix corresponding to $\pi$ in the irreducible $\aS_n$-representation (the so-called \emph{Specht module}) of type~$\la$.
From this definition it immediately follows that $\chi^\la(\pi)=\chi^\la(\sigma)$ if $\pi$ and $\sigma$ are permutations that have the same cycle type $\mu$, and we use this fact to define $\chi^\la(\mu)$ for a partition~$\mu$.

For a composition $\a$ of $n$, consider the Young subgroup $\aS_{\a}:=\aS_{a_1}\times \aS_{a_2}\times \ldots$ of $\aS_n$, where $\aS_{a_1}$ permutes only $\{1,\ldots, a_1\}$, $S_{a_2}$ permutes only $\{a_1+1,\ldots,a_1+a_2\}$, etc. The induced trivial representation $\ind_{\aS_{\a}}^{\aS_n} 1$ can be defined as the action of $\aS_n$ on the left cosets of $\aS_n/\aS_{\a}$, see~\cite[$\S$7.18]{EC2}. This is equivalent to the action of $\aS_n$ on words with $a_1$ many $1$s, $a_2$ many $2$s etc by permuting their positions. Denote by $\phi^{\a}$ the character of this representation, then $\phi^{\a}(\pi) = \#\{ u \mid u \pi =u\}$, the number of words fixed by $\pi$. A word $u$ is fixed by $\pi$ if and only if $u_i=u_j$ for all $i,j$ in the same cycle of $\pi$. Thus the number of fixed words is equal to the number of ways we can label the cycles of $\pi$ with $1,2,\ldots$, so that the total number of elements in the cycles labeled by $i$ is equal to $a_i$.

The \emph{Frobenius character formula}, see e.g.\ \cite[Eq.~2.3.8]{JK} (equivalent to the
\emph{Jacobi--Trudi identity}, see e.g.\ \cite[$\S$7.16 and $\S$7.18]{EC2}), gives
\begin{equation}\label{eq:char_jt}
\chi^\la \, = \, \sum_{\sigma \in \aS_{\ell(\la)}} \sgn(\sigma)  \phi^{\la+\sigma-\id}.
\end{equation}
Here $\id=(1,2,\ldots,\ell)\in \aS_\ell$ is the identity permutation, and
$(\la+\sigma-\id)$  denotes the composition
$(\la_1 + \sigma_1 -1, \la_2 + \sigma_2 - 2,\ldots, \la_\ell + \sigma_\ell - \ell)$.
Also, in \eqref{eq:char_jt}, for a composition $\a$ in the summation we let $\phi^{\a}:=0$ if $a_i<0$ for some $i$,
 $\phi^{(\a,0,\b)}:=\phi^{(\a,\b)}$, and $\phi^{(0)}:=1$.

\medskip

\section{Computational Complexity}

\subsection{$\CeqP$ and the Collapse of the Polynomial Hierarchy}
\label{sec:CeqP}
We will use well-known complexity classes
with oracle access to a language in the standard way, see e.g.~\cite{Pap94}.
As it is common, the oracle language is written in the exponent.
For a function $f:\{0,1\}^*\to\IZ$ and an integer comparison operator $\sim$ we define the language $[f\sim 0] := \{w \in \{0,1\}^* \mid f(w)\sim 0\}$.

For a class $Z$ of functions $\{0,1\}^*\to\IZ$ and an integer comparison operator $\sim$ we define the decision class $[Z\sim 0] \subseteq 2^{\{0,1\}^*}$ via: $L \in [Z\sim 0]$ if and only if there exists $f\in Z$ with the property that for all $w\in\{0,1\}^*$ we have $w \in L$ if and only if $f(w)\sim 0$.
Using this notation, we recall that
$\NP = [\sharpP > 0]$,
$\coNP = [\sharpP = 0]$,
$\CeqP = [\GapP = 0]$, and
$\coCeqP = [\GapP \neq 0]$.
In particular, $\coNP\subseteq\CeqP$.

Recall that
$\Sigmap_{0} = \P$, $\Sigmap_{i+1} = \NP^{\Sigmap_i}$, and that $\PH = \bigcup_{i \in \IN} \Sigmap_{i}$.
Moreover, for a class $A \subseteq 2^{\{0,1\}^*}$,
recall that the complement class $\ComCla{co}A$ is defined via
$L \in \ComCla{co}A$ if and only if $\overline{L} \in A$.
For a language~$L$ we write $\langle L \rangle$ to be the class of all languages that are many-one reducible to $L$, for example $\NP=\langle\textup{3SAT}\rangle$, where 3SAT is the language of all satisfiable Boolean formulas in 3CNF.
A language $L\subseteq\{0,1\}^*$ is called $\CeqP$-hard
under many-one reductions if $\CeqP \subseteq \langle L\rangle$.
Our main application is the case where $L=[f=0]$, where $f$ is a function in $\GapP$.

\begin{proposition}\label{pro:collapse}
Given a function $f:\{0,1\}^*\to\IZ$ with the property that $[f=0]$ is $\CeqP$-hard under many-one reductions.
Fix a function $q: \mathbb{Z} \to \mathbb{Z}$ such that $q(0)=0$ and $q(x)>0$ for all $x>0$, for example $q(x)=x^2$ or $q(x)=|x|$.
If $q(f) \in \sharpP$, then $\coNP=\CeqP$ (and in particular $\PH=\Sigmap_2$).
\end{proposition}
\begin{proof} Note that
$\coNP\subseteq\CeqP$ by definition. For the other direction, observe that\\
$\CeqP \subseteq \langle[f=0]\rangle = \ComCla{co}\langle[f\neq0]\rangle = \ComCla{co}\langle[q(f)\neq0]\rangle = \ComCla{co}\langle[q(f)>0]\rangle \stackrel{q(f)\in\sharpP}{\subseteq} \coNP$.
\end{proof}

\smallskip

Since $\ComCla{PH} \subseteq \NP^{\CeqP}$ (see (\cite{Tar91}, and also \cite{Gre93}),
we have that if $\coNP=\CeqP$, then $\PH \subseteq \NP^\CeqP = \NP^\coNP = \Sigmap_2$, and hence $\Sigmap_2=\PH$.
As an aside, we remark that if $[f=0]$ is $\CeqP$-hard \emph{under Turing-reductions} only, then $q(f)\in\sharpP$ also implies $\PH=\Sigmap_2$ via
\[
\PH \ \subseteq \ \NP^\CeqP \ \subseteq \ \NP^{\P^{[f=0]}} \ = \ \NP^{[f=0]} \ = \ \NP^{[q(f)\neq 0]} \ \stackrel{q(f)\in\sharpP}{\subseteq} \ \NP^\coNP \ = \ \Sigmap_2.
\]

\subsection{3D- and 4D-matchings}\label{ss:3DM}
Recall the following standard counting problems, see \cite{GJ79}.

\bigskip

\noindent\begin{minipage}{\textwidth}
Problem   $\problem{\#CircuitSAT}$:
\begin{compactitem}
 \item Input: A Boolean circuit $C$ with $n$ inputs.
 \item Output: The number of $w\in\{0,1\}^n$ with $C(w)=$ true.
\end{compactitem}
\end{minipage}

\bigskip

\noindent\begin{minipage}{\textwidth}
For $i \in \{2,3,4\}$, Problem $\sharpiDM$, 
\begin{compactitem}
 \item Input: A subset $E \subseteq [k]^i$.
 \item Output: The number of $M\in\binom{E}{k}$ such that  $\forall \{(x_1,\ldots,x_i),(x'_1,\ldots,x'_i)\} \in \binom{M}{2}$ we have: $x_1\neq x_1', \ldots, x_i\neq x_i'$.
\end{compactitem}
\end{minipage}

\bigskip


\newcommand{\oneinthree}{\textup{one-in-three}}
We will use a well-known parsimonious polynomial-time reduction $R$ from $\problem{\#CircuitSAT}$ to $\sharpIIIDM$ as a black-box \footnote{Since it is difficult to find this exact statement in the literature, we here give pointers on how to obtain the reduction as a composition of three parsimonious polynomial time reductions. First, take
the classical \emph{Tseytin transformation} (see e.g.\ Example~$8.3$ in~\cite[page 163]{Pap94}), which
is a parsimonious polynomial time reduction from \problem{\#CircuitSAT} to \problem{\#3SAT}.
Next, take Schaefer's parsimonious reduction \cite{Sch78} from \problem{\#3SAT} to \problem{\#1-in-3SAT}:
replace $x\vee y \vee z$ by $\oneinthree(\neg x,u_1,u_2)\,\wedge\,\oneinthree(y,u_2,u_3)\,\wedge\,\oneinthree(\neg z,u_3,u_4)$.
Finally, take Young's parsimonious reduction from $\sharpIIIDM$ to \problem{\#1-in-3SAT}, defined
via a promise problem called \problem{1+3DM}, see~\cite{You20}.
}.

\subsection{Ordered set partitions}
Let $\a=(a_1,\ldots,a_m)$ be a positive integer sequence and $\b=(b_1,\ldots,b_\ell)$ be a nonnegative
integer sequence, both with the same total sum: $|\a|=|\b|$.
An \emph{ordered set partition with item sizes~$\a$ and bin sizes~$\b$}
is a tuple $\vv K = (K_1,\ldots,K_\ell)$ of pairwise disjoint subsets $K_1,\ldots,K_\ell \subseteq [m]$, such that
\begin{equation}
\label{eq:setpartitionconditions}
\bigcup_{i=1}^\ell K_i  =  [m]  \quad \text{and} \quad
\sum_{j\in K_i} a_j  =  b_i  \quad \text{for all $1\le i \le \ell$. }
\end{equation}
We use $\cP(\a,\b)$ to denote the set of ordered set partitions
with with item sizes~$\a$ and bin sizes~$\b$, and let $\rP(\a,\b)=\bigl|\cP(\a,\b)\bigr|$. Here we will assume that $a_i>0$ for all $i$ and $P(\a,\b) =0$ if $b_i<0$ for some $i$. All set partitions considered here will be ordered.

\bigskip

\noindent\begin{minipage}{\textwidth}
Problem \problem{\#SetPartition}:
\begin{compactitem}
 \item Input: $(\a,\b) \in \IN^\ell\times\IN^m$.
 \item Output: The number of $\vv K$ that satisfy \eqref{eq:setpartitionconditions}.
\end{compactitem}
\end{minipage}

\bigskip

\noindent In other words, $\problem{\#SetPartition}(\a,\b)=\rP(\a,\b)$.

\medskip

\section{Main result}
In this section we prove \Cref{thm:CeqPhardnessofchar} and \Cref{thm:PH-char}.
Combined with \Cref{pro:collapse}, \Cref{thm:CeqPhardnessofchar} immediately implies \Cref{thm:main}.

\subsection{Characters and set partitions} We start by translating our problem
from the language of characters of $\aS_n$ into the language of ordered set partitions.

\begin{lemma}\label{l:phi_part}
The characters of the induced representation $\phi^\nu$ evaluated at a conjugacy class of type $\al$ are equal to the number of ordered set partitions of $\al$ into sets of sizes $\nu$. That is,
$$
\phi^\nu(\al) = \rP(\al,\nu).$$
\end{lemma}
\begin{proof}
As explained in $\S$\ref{ss:rep_theory}, the evaluation $\phi^\nu(\al)$ is equal to the number of words $u$ with $\nu_i$ letters $i$ for $i=1,\ldots,\ell(\nu)$, which are fixed under permuting the positions of their entries by a permutation $\pi$ of cycle type $\al=(\al_1,\ldots, \al_m)$. Thus, the positions (elements of $\pi$) in the same cycle have the same letter. Let the cycles of $\pi$ be $c_1,\ldots,c_m$ of lengths $\al_1,\ldots,\al_m$ respectively. Let $K_i=\{ j: u|_{c_j}=(i,\ldots,i)\}$ be the set of cycles on which $u$ has value $i$. Then $(K_1,\ldots,K_{\ell(\nu)})$ is an ordered set partition with item sizes $\al_1,\al_2,\ldots$ and bin sizes $\nu_1,\nu_2,\ldots$ Conversely, such a set partition determines the word $u$ uniquely, and so $\rP(\al,\nu) = \phi^\nu(\al)$.
\end{proof}

\begin{proposition}\label{p:char-part}
Let $\la \vdash n$, $l\geq\ell(\la)$, and let $\al$ be a composition of $n$ with at most $l$ parts. Then
$$\chi^\la(\al)  =  \sum_{\sigma \in \aS_{l}} \sgn(\sigma)  \rP(\al,\la+\sigma-\id).$$
\end{proposition}
\begin{proof}
This follows directly from equation~\eqref{eq:char_jt} and Lemma~\ref{l:phi_part}.
\end{proof}

\begin{lemma}\label{p:char-part2}
Let $\a$ and $\b$ be two positive sequences with equal sums, and let $\b$ have $\ell$ parts. Let $p\geq\ell+1$, $\la=\sort(p\cdot \b)$  and $\al =p\cdot \a +(1,-1,0,\ldots)$.
Then
\begin{equation}\label{eq:char-part2-sums}
\chi^\la(\al) \, = \, \sum_{i=1}^\ell \rP\big(\ov{\a}, \b -(a_1+a_2)e_i\big)  -
\sum_{i=1}^{\ell-1} \rP\big(\ov{\a},\b - a_1e_i -a_2e_{i+1}\big),
\end{equation}
where
$e_i$ is the $i$-th standard basis vector 
and
$\bar{\a} = (a_3,a_4,\ldots)$.
\end{lemma}
\begin{proof}
Without loss of generality, assume $b_1 \geq b_2 \geq \ldots$
We apply \Cref{p:char-part} with the given partitions. Consider a set partition of item sizes $\al=(pa_1+1,pa_2-1,pa_3,\ldots)$ into bins of sizes $pb_i +\sigma_i -i$ for $i=1,\ldots,\ell$.
Since $p$ divides $\al_i$ for $i \neq 1,2$, we must have that at most two of the sum sets are not divisible by $p$, and so $\sigma_j \equiv j \pmod{p}$ for all but possibly two values of $j$ corresponding to the bins containing $\al_1$ and $\al_2$.

We have two possibilities. In the first case, both $\al_1, \al_2$ are in the same set (bin), of size $\la_i + \sigma_i-i$ for some~$i$. Since $\al_1+\al_2 =p(a_1+a_2)$ and $\al_i=pa_i$ for all other $i$s, the bin size must be divisible by $p$. Thus $0 \equiv \la_i + \sigma_i-i \equiv pb_i +\sigma_i -i \pmod{p}$ for all $i$ and so $\sigma = \id$. Choosing in which set the $\al_1+\al_2$ go gives us the left big summation in \eqref{eq:char-part2-sums}.

In the second case, $\al_1,\al_2$ are in two different sets (bins), say $t$ and $r$, whose sums must then be $\equiv 1, -1 \pmod{p}$ respectively. Since all other item sizes are divisible by $p$, we must have $\la_i+\sigma_i-i = pb_i +\sigma_i-i \equiv 0 \pmod{p}$ for $i\neq r,t$.
Thus $\sigma_i=i$ for $i\neq t,r$ and we must have $\sigma_t=r$ and $\sigma_r=t$. Then $\la_t+r-t \equiv +1 \pmod{p}$ and $\la_r +t-r \equiv -1 \pmod{p}$. Since $1 \le r,t \le p-1$, we must have $r=t+1$, and we arrive in the other big summation, with $\al_1$ in set $t$ and $\al_2$ in set $t+1$. This completes the proof.
\end{proof}

\begin{proposition}\label{p:char-part3}
Let $\c$ and $\d$ be two sequences of nonnegative integers, such that $|\c|=|\d|+6$. Then there are polynomial time computable partitions $\la$ and $\al$, such that
\[
\chi^\la(\al)  =  \rP\big(\c, (2,4,\d)\big)  -  \rP\big(\c,(1,5,\d)\big).
\]
\end{proposition}
\begin{proof}
We will use \Cref{p:char-part2} with the following construction.
Set $m := \max\{c_1,\ldots, d_1,\ldots\}+4$.
Let $\a := (2, m , m-3,  \c)$ and $\b := (m+4,m+1, \d)$.  Now construct $\la=\sort(p\cdot\b)=(p(m+4),p(m+1),p\cdot\sort(\d))$ and~$\al=p\cdot\a+(1,-1)$ as in \Cref{p:char-part2}. Note that $b_i \geq a_1+a_2 = m+2$ only for $i=1$, and $b_{i+1} \geq m=a_2$ only for $i=1$, so the only nonzero terms in equation~\eqref{eq:char-part2-sums} are the summands for $i=1$.

We thus obtain
$$\chi^\la(\al) \, = \, \rP\big( (m-3, \c), (2, m+1, \d) \big)  -  \rP\big( (m-3, \c), (m+2, 1, \d) \big).
$$
Since $m-3 > d_i$ for all $i$, the item of size $m-3$ can only go into the bins of sizes $m+1$ and $m+2$, respectively.  Therefore,
we have:
$$\chi^\la(\al) \, = \, \rP\big( \c, (2, 4, \d) \big)  -  \rP\big( \c, (5,1,\d)\big),
$$
and the proof is complete.
\end{proof}

\subsection{The join of two 3D-matchings}
Let $[k] := \{1,\ldots,k\}$.
In particular, $[u]\subseteq[k]$ for $u \leq k$.
We write $+_k$ to indicate addition modulo $k$ in $[k]$, e.g., $k+_k 1=1$.
We consider uniform 3-partite hypergraphs on a vertex set $V_1 \cup V_2 \cup V_3$, where $V_i=\{(1,i),(2,i),\ldots,(k,i)\}$, which we identify as copies of $[k]$ for simplicity. Let $E\subset V_1 \times V_2 \times V_3$ be a set of hyperedges each covering 3 vertices, one from each set $V_i$, and encoded as triples $(e_1,e_2,e_3)$ where $e_i \in [k]$ by ignoring the second coordinates of the vertices, so we can think of $E \subseteq [k]^3$.
For $u \geq k$ define the \emph{padding} $E_{u}\subseteq [u]^3$
via $E_{u} := E \cup \{(x,x,x) \mid k<x\leq u\}$.
Clearly $\sharpIIIDM(E) = \sharpIIIDM(E_{u})$ for every $u \geq k$.
This is illustrated as the first step in Figure~\ref{fig:join}.
Given two subsets $E\subseteq[k]^3$ and $E' \subseteq[k']^3$, let $u:=1+\max\{k,k'\}$.
We define the $\join(E,E'):=(J,H,H')$ to be the following 3-tuple $(J,H,H')$, consisting of a $\sharpIVDM$ instance $J\subseteq[u]^4$ and two elements $H\in J$ and $H'\in J$ as follows.
\begin{itemize}
 \item $J := \{(x,x,y,z) \mid (x,y,z)\in E_u\} \ \cup \ \{(x+_u 1,x,y,z) \mid (x,y,z)\in E'_u\}$,
 \item $H := (u,u,u,u)$,
 \item $H' := (1,u,u,u)$.
\end{itemize}
This construction is illustrated as the last step in Figure~\ref{fig:join}.
Elements in $J$ are called hyperedges.
Note that by construction $H$ and $H'$ are the only hyperedges in $J$ that have $u$ as their last coordinate.

\tikzset{
  double arrow/.style args={#1 colored by #2 and #3}{
    Round Cap-Round Cap, shorten <=-0.5cm, shorten >=-0.5cm, line width=#1,#2,
    postaction={draw,Round Cap-Round Cap,#3,line width=3*(#1)/4,shorten <=-0.41cm,shorten >=-0.41cm},
  }
}

\begin{figure}
\begin{center}
\scalebox{0.5}{
\begin{tikzpicture}
\node at (0,0.5) {
\begin{tikzpicture}[y=-1cm]
\fill[fill=blue!10, rounded corners] (-1,-1) -- (7,-1) -- (7,3) -- (-1,3) -- cycle;
\draw[double arrow=15pt colored by black and gray!20, rounded corners] (0,2)--(3,2)--(6,0);
\draw[double arrow=15pt colored by black and gray!20, rounded corners] (0,2)--(3,0)--(6,0);
\draw[double arrow=15pt colored by black and gray!20, rounded corners] (0,0)--(3,2)--(6,2);
\node[circle, draw=black, fill=red!80, line width=0.7mm, inner sep=0pt,minimum size=7pt] at (0,0) {};
\node[circle, draw=black, fill=red!80, line width=0.7mm, inner sep=0pt,minimum size=7pt] at (0,2) {};
\node[circle, draw=black, fill=green!80, line width=0.7mm, inner sep=0pt,minimum size=7pt] at (3,0) {};
\node[circle, draw=black, fill=green!80, line width=0.7mm, inner sep=0pt,minimum size=7pt] at (3,2) {};
\node[circle, draw=black, fill=blue!80, line width=0.7mm, inner sep=0pt,minimum size=7pt] at (6,0) {};
\node[circle, draw=black, fill=blue!80, line width=0.7mm, inner sep=0pt,minimum size=7pt] at (6,2) {};
\end{tikzpicture}
};
\node at (9,0.5) {
\begin{tikzpicture}[y=-1cm]
\fill[fill=blue!10, rounded corners] (-1,-1) -- (7,-1) -- (7,5) -- (-1,5) -- cycle;
\draw[double arrow=25pt colored by black and gray!40, rounded corners] (0,4)--(3,0)--(6,4);
\draw[double arrow=25pt colored by black and gray!40, rounded corners] (0,0)--(3,2)--(6,4);
\draw[double arrow=25pt colored by black and gray!40, rounded corners] (0,2)--(3,4)--(6,0);
\draw[double arrow=25pt colored by black and gray!40, rounded corners] (0,0)--(3,0)--(6,2);
\node[circle, draw=black, fill=red!80, line width=0.7mm, inner sep=0pt,minimum size=7pt] at (0,0) {};
\node[circle, draw=black, fill=red!80, line width=0.7mm, inner sep=0pt,minimum size=7pt] at (0,2) {};
\node[circle, draw=black, fill=red!80, line width=0.7mm, inner sep=0pt,minimum size=7pt] at (0,4) {};
\node[circle, draw=black, fill=green!80, line width=0.7mm, inner sep=0pt,minimum size=7pt] at (3,0) {};
\node[circle, draw=black, fill=green!80, line width=0.7mm, inner sep=0pt,minimum size=7pt] at (3,2) {};
\node[circle, draw=black, fill=green!80, line width=0.7mm, inner sep=0pt,minimum size=7pt] at (3,4) {};
\node[circle, draw=black, fill=blue!80, line width=0.7mm, inner sep=0pt,minimum size=7pt] at (6,0) {};
\node[circle, draw=black, fill=blue!80, line width=0.7mm, inner sep=0pt,minimum size=7pt] at (6,2) {};
\node[circle, draw=black, fill=blue!80, line width=0.7mm, inner sep=0pt,minimum size=7pt] at (6,4) {};
\end{tikzpicture}
};
\node at (0,-3.25) {\scalebox{3}{$\Downarrow$}};
\node at (9,-3.25) {\scalebox{3}{$\Downarrow$}};
\node at (0,-8) {
\begin{tikzpicture}[y=-1cm]
\fill[fill=blue!10, rounded corners] (-1,-1) -- (7,-1) -- (7,7) -- (-1,7) -- cycle;
\draw[double arrow=15pt colored by black and gray!20, rounded corners] (0,6)--(3,6)--(6,6);
\draw[double arrow=15pt colored by black and gray!20, rounded corners] (0,4)--(3,4)--(6,4);
\draw[double arrow=15pt colored by black and gray!20, rounded corners] (0,2)--(3,2)--(6,0);
\draw[double arrow=15pt colored by black and gray!20, rounded corners] (0,2)--(3,0)--(6,0);
\draw[double arrow=15pt colored by black and gray!20, rounded corners] (0,0)--(3,2)--(6,2);
\node[circle, draw=black, fill=red!80, line width=0.7mm, inner sep=0pt,minimum size=7pt] at (0,0) {};
\node[circle, draw=black, fill=red!80, line width=0.7mm, inner sep=0pt,minimum size=7pt] at (0,2) {};
\node[circle, draw=black, fill=red!80, line width=0.7mm, inner sep=0pt,minimum size=7pt] at (0,4) {};
\node[circle, draw=black, fill=red!80, line width=0.7mm, inner sep=0pt,minimum size=7pt] at (0,6) {};
\node[circle, draw=black, fill=green!80, line width=0.7mm, inner sep=0pt,minimum size=7pt] at (3,0) {};
\node[circle, draw=black, fill=green!80, line width=0.7mm, inner sep=0pt,minimum size=7pt] at (3,2) {};
\node[circle, draw=black, fill=green!80, line width=0.7mm, inner sep=0pt,minimum size=7pt] at (3,4) {};
\node[circle, draw=black, fill=green!80, line width=0.7mm, inner sep=0pt,minimum size=7pt] at (3,6) {};
\node[circle, draw=black, fill=blue!80, line width=0.7mm, inner sep=0pt,minimum size=7pt] at (6,0) {};
\node[circle, draw=black, fill=blue!80, line width=0.7mm, inner sep=0pt,minimum size=7pt] at (6,2) {};
\node[circle, draw=black, fill=blue!80, line width=0.7mm, inner sep=0pt,minimum size=7pt] at (6,4) {};
\node[circle, draw=black, fill=blue!80, line width=0.7mm, inner sep=0pt,minimum size=7pt] at (6,6) {};
\end{tikzpicture}
};
\node at (9,-8) {
\begin{tikzpicture}[y=-1cm]
\fill[fill=blue!10, rounded corners] (-1,-1) -- (7,-1) -- (7,7) -- (-1,7) -- cycle;
\draw[double arrow=25pt colored by black and gray!40, rounded corners] (0,6)--(3,6)--(6,6);
\draw[double arrow=25pt colored by black and gray!40, rounded corners] (0,4)--(3,0)--(6,4);
\draw[double arrow=25pt colored by black and gray!40, rounded corners] (0,0)--(3,2)--(6,4);
\draw[double arrow=25pt colored by black and gray!40, rounded corners] (0,2)--(3,4)--(6,0);
\draw[double arrow=25pt colored by black and gray!40, rounded corners] (0,0)--(3,0)--(6,2);
\node[circle, draw=black, fill=red!80, line width=0.7mm, inner sep=0pt,minimum size=7pt] at (0,0) {};
\node[circle, draw=black, fill=red!80, line width=0.7mm, inner sep=0pt,minimum size=7pt] at (0,2) {};
\node[circle, draw=black, fill=red!80, line width=0.7mm, inner sep=0pt,minimum size=7pt] at (0,4) {};
\node[circle, draw=black, fill=red!80, line width=0.7mm, inner sep=0pt,minimum size=7pt] at (0,6) {};
\node[circle, draw=black, fill=green!80, line width=0.7mm, inner sep=0pt,minimum size=7pt] at (3,0) {};
\node[circle, draw=black, fill=green!80, line width=0.7mm, inner sep=0pt,minimum size=7pt] at (3,2) {};
\node[circle, draw=black, fill=green!80, line width=0.7mm, inner sep=0pt,minimum size=7pt] at (3,4) {};
\node[circle, draw=black, fill=green!80, line width=0.7mm, inner sep=0pt,minimum size=7pt] at (3,6) {};
\node[circle, draw=black, fill=blue!80, line width=0.7mm, inner sep=0pt,minimum size=7pt] at (6,0) {};
\node[circle, draw=black, fill=blue!80, line width=0.7mm, inner sep=0pt,minimum size=7pt] at (6,2) {};
\node[circle, draw=black, fill=blue!80, line width=0.7mm, inner sep=0pt,minimum size=7pt] at (6,4) {};
\node[circle, draw=black, fill=blue!80, line width=0.7mm, inner sep=0pt,minimum size=7pt] at (6,6) {};
\end{tikzpicture}
};
\node at (2.5, -12.5) {\scalebox{3}{\rotatebox{45}{$\Downarrow$}}};
\node at (6.5, -12.5) {\scalebox{3}{\rotatebox{-45}{$\Downarrow$}}};
\node at (4.5, -17) {
\begin{tikzpicture}[y=-1cm]
\fill[fill=blue!10, rounded corners] (-4,-1) -- (7,-1) -- (7,7) -- (-4,7) -- cycle;
\draw[double arrow=25pt colored by black and gray!40, rounded corners] (-3,0) -- (0,6)--(3,6)--(6,6);
\draw[double arrow=25pt colored by black and gray!40, rounded corners] (-3,6) -- (0,4)--(3,0)--(6,4);
\draw[double arrow=25pt colored by black and gray!40, rounded corners] (-3,2) -- (0,0)--(3,2)--(6,4);
\draw[double arrow=25pt colored by black and gray!40, rounded corners] (-3,4) -- (0,2)--(3,4)--(6,0);
\draw[double arrow=25pt colored by black and gray!40, rounded corners] (-3,2) -- (0,0)--(3,0)--(6,2);
\draw[double arrow=15pt colored by black and gray!20, rounded corners] (-3,6) -- (0,6)--(3,6)--(6,6);
\draw[double arrow=15pt colored by black and gray!20, rounded corners] (-3,4) -- (0,4)--(3,4)--(6,4);
\draw[double arrow=15pt colored by black and gray!20, rounded corners] (-3,2) -- (0,2)--(3,2)--(6,0);
\draw[double arrow=15pt colored by black and gray!20, rounded corners] (-3,2) -- (0,2)--(3,0)--(6,0);
\draw[double arrow=15pt colored by black and gray!20, rounded corners] (-3,0) -- (0,0)--(3,2)--(6,2);
\node[circle, draw=black, fill=orange!60, line width=0.7mm, inner sep=0pt,minimum size=7pt] at (-3,0) {};
\node[circle, draw=black, fill=orange!60, line width=0.7mm, inner sep=0pt,minimum size=7pt] at (-3,2) {};
\node[circle, draw=black, fill=orange!60, line width=0.7mm, inner sep=0pt,minimum size=7pt] at (-3,4) {};
\node[circle, draw=black, fill=orange!60, line width=0.7mm, inner sep=0pt,minimum size=7pt] at (-3,6) {};
\node[circle, draw=black, fill=red!80, line width=0.7mm, inner sep=0pt,minimum size=7pt] at (0,0) {};
\node[circle, draw=black, fill=red!80, line width=0.7mm, inner sep=0pt,minimum size=7pt] at (0,2) {};
\node[circle, draw=black, fill=red!80, line width=0.7mm, inner sep=0pt,minimum size=7pt] at (0,4) {};
\node[circle, draw=black, fill=red!80, line width=0.7mm, inner sep=0pt,minimum size=7pt] at (0,6) {};
\node[circle, draw=black, fill=green!80, line width=0.7mm, inner sep=0pt,minimum size=7pt] at (3,0) {};
\node[circle, draw=black, fill=green!80, line width=0.7mm, inner sep=0pt,minimum size=7pt] at (3,2) {};
\node[circle, draw=black, fill=green!80, line width=0.7mm, inner sep=0pt,minimum size=7pt] at (3,4) {};
\node[circle, draw=black, fill=green!80, line width=0.7mm, inner sep=0pt,minimum size=7pt] at (3,6) {};
\node[circle, draw=black, fill=blue!80, line width=0.7mm, inner sep=0pt,minimum size=7pt] at (6,0) {};
\node[circle, draw=black, fill=blue!80, line width=0.7mm, inner sep=0pt,minimum size=7pt] at (6,2) {};
\node[circle, draw=black, fill=blue!80, line width=0.7mm, inner sep=0pt,minimum size=7pt] at (6,4) {};
\node[circle, draw=black, fill=blue!80, line width=0.7mm, inner sep=0pt,minimum size=7pt] at (6,6) {};
\end{tikzpicture}
};
\end{tikzpicture}
}
\end{center}

\caption{
On the top left: The \problem{\#3DM} instance $E=\{(1,2,2),(2,1,1),(2,2,1)\}$, where $V_1$ are the red circles in the first column, $V_2$ are the green circles in the second column, and $V_3$ are the blue circles in the third column.
For example, $(1,2,2)$ is depicted as a hyperedge containing the red vertex in column~1, row~1, and the green vertex in column~2, row~2, and the blue vertex in column~3, row~2.
On the top right: The \problem{\#3DM} instance $\{(1,1,2),(1,2,3),(2,3,1),(3,1,3)\}$.
The join of these two instances is obtained by first padding them to the same number of rows $u$ (here $u=4$), and then adding another dimension to each hyperedge (which adds a column of points at the front) and taking the union of both hypergraphs. The two special hyperedges $H=(4,4,4,4)$ and $H'=(1,4,4,4)$ are the ones containing the bottom right vertex. The different shades of gray for the hyperedges are just for illustration.}
\label{fig:join}
\end{figure}

\begin{lemma}
\label{lem:joinminusH}
Given two subsets $E\subseteq[k]^3$ and $E' \subseteq[k']^3$, let $(J,H,H')=\join(E,E')$.
Then
$\sharpIVDM(J\setminus\{H'\}) = \sharpIIIDM(E)$ and
$\sharpIVDM(J\setminus\{H\}) = \sharpIIIDM(E')$.
\end{lemma}
\begin{proof}
Clearly $\sharpIVDM(J\setminus\{H'\}) \geq \sharpIIIDM(E)$,
because a 3D-matching $M \subseteq E$ can be converted to a 4D-matching by converting each hyperedge $(x,y,z)$ to $(x,x,y,z)$, and adding the special hyperedge $(u,u,u,u)$.
Analogously one shows $\sharpIVDM(J\setminus\{H\}) \geq \sharpIIIDM(E')$.

The reverse is also true, which can be seen as follows.
If we restrict each hyperedge in $J$ to their first two coordinates, we get the so-called \emph{cycle graph} $C_{2u}$ on $2u$ vertices, which has the crucial property that $\sharpIIDM(C_{2u})=2$:

\begin{center}
\begin{tikzpicture}
\draw (0,4) -- (1,4);
\draw (0,3) -- (1,3);
\draw (0,2) -- (1,2);
\draw (0,0) -- (1,0);
\draw (1,4) -- (0,3);
\draw (1,3) -- (0,2);
\draw (1,2) -- (0,1);
\draw (1,1) -- (0,0);
\node[fill=orange!60, draw=black, circle, inner sep = 0cm, minimum size=0.2cm] at (0,4) {};
\node[fill=orange!60, draw=black, circle, inner sep = 0cm, minimum size=0.2cm] at (0,3) {};
\node[fill=orange!60, draw=black, circle, inner sep = 0cm, minimum size=0.2cm] at (0,2) {};
\node[fill=orange!60, draw=black, circle, inner sep = 0cm, minimum size=0.2cm] at (0,0) {};
\node[fill=red!80, draw=black, circle, inner sep = 0cm, minimum size=0.2cm] at (1,4) {};
\node[fill=red!80, draw=black, circle, inner sep = 0cm, minimum size=0.2cm] at (1,3) {};
\node[fill=red!80, draw=black, circle, inner sep = 0cm, minimum size=0.2cm] at (1,2) {};
\node[fill=red!80, draw=black, circle, inner sep = 0cm, minimum size=0.2cm] at (1,0) {};
\node[fill=white] at (0,1) {$\vdots$};
\node[fill=white] at (1,1) {$\vdots$};
\draw (1,0) -- (0,4);
\node[fill=orange!60, draw=black, circle, inner sep = 0cm, minimum size=0.2cm] at (0,4) {};
\node[fill=red!80, draw=black, circle, inner sep = 0cm, minimum size=0.2cm] at (1,0) {};
\end{tikzpicture}
\end{center}

Since every 4D-matching restricted to the first two coordinates is a 2D-matching on the restricted graph $C_{2u}$,
every 4D-matching contains either only hyperedges
with
(first,{}second) coordinate $(x,x)$
or only hyperedges with
with (first,{}second) coordinate $(x+_u 1,x)$.
Hence, if for a 4D-matching that contains $H$ we delete the first coordinates of all hyperedges in the matching, the result is a 3D-matching of $E_u$. Analogously for $H'$.
\end{proof}

\subsection{An auxiliary \problem{SetPartition} instance}
\label{subsec:auxsetpartition}

We follow the ideas of the proof of the strong NP-hardness of 4-partition in \cite[p.~96]{GJ79} rather closely.

For a hyperedge $e = (e_1,e_2,e_3) \in [k]^3$ let $\varphi(e)=\{(e_1,1),(e_2,2),(e_3,3)\}$ be the set of vertices covered by $e$ in depictions such as Figure~\ref{fig:join}.
Analogously for elements in $[k]^4$.
For $E \subseteq [k]^3$ let $\varphi(E) := \bigcup_{e\in E}\varphi(e)$.
Given two subsets $E\subseteq[k]^3$ and $E' \subseteq[k']^3$ with $\varphi(E)=[k]\times[3]$ and $\varphi(E')=[k']\times[3]$, i.e. the edges cover all vertices.
let $\join(E,E')=(J,H,H')$ with $J \subset [u]^4$.
Note that $\varphi(J)=[u]\times[4]$.
We now describe how from $(J,H,H')$ one constructs a \problem{SetPartition} instance
$(\a,\b)$, as described in \cite{GJ79}, with properties described in \Cref{lem:auxplacement}.

Given $(i,j)\in[u]\times[4]$, let $\mult_J(i,j):=\{e \in J \mid (i,j)\in \varphi(e)\}$ denote the number of hyperedges that that have value $i$ at position $j$, or, pictorially as in Figure~\ref{fig:join}, they contain the vertex $(i,j)$.
%

We streamline the construction in \cite{GJ79} a bit: We choose an $r \geq 2\cdot(\max\{5,u\}\cdot (4+5|J|) +1)$ (this exact bound has no further relevance) and write numbers in $r$-adic notation $[a_1,a_2,a_3,\ldots] := a_1 r + a_2 r^2 + a_3 r^3 + \ldots$ with the peculiarity that the constant term is zero (this will be changed in \S\ref{subsec:modified}).
Inside this notation we use the shorthand $0^j$ to denote a sequence of $j$ many zeros. We also use the shorthand
$\eta_j$ to denote the sequence $(0^{j-1},1,0^{4-j}) \in \{0,1\}^4$, i.e., the $j$-th standard basis vector in 4 dimensions.
We write $a\cdot\eta_j = (0^{j-1},a,0^{i-j}) \in \{0,a\}^4$.
We say that $a_i$ is the \emph{$i$-th coefficient}.
We will have $\forall i\neq 0: 0 \leq a_i< r$, and only in \S\ref{subsec:modified} we allow for exactly two negative $a_0$, for the ease of presentation.
We chose $r$ large enough such that all additions of numbers in $r$-adic notation in \S\ref{subsec:auxsetpartition} will have no $r$-adic carry.
Let $\beta(j):=4$ if $1\leq j \leq 3$ and $\beta(4):=0$.

Let there be $|J|$ many bins in this \problem{SetPartition} instance, and let all the bins have the same size, the bin size given by
$b_1:=[1,1,1,1,1,u,u,u,u,12]$, in other words $\b := (\underbrace{b_1,b_1,\ldots,b_1}_{|J| \text{ times}})$.
The items are created as follows.
\begin{itemize}
 \item For each $(i,j) \in [u] \times [4]$ we create an item of size
 $[\eta_j,0,i\cdot \eta_j,3]$.
  These are called \emph{real vertex items}.
 \item For each $(i,j) \in [u] \times [4]$ we create $\mult_J(i,j)-1$ many items of size $[\eta_j,0,i\cdot \eta_j,\beta(j)]$.
 These are called \emph{dummy vertex items}.
 Here we used that $\mult(i,j)\geq 1$, which is guaranteed, because $\varphi(J)=[u]\times[4]$.
 \item For each hyperedge $(w,x,y,z) \in J$ we create an item of size $[0^4,1,u-w,u-x,u-y,u-z,0]$.
  These are called \emph{hyperedge items}.
 \end{itemize}
This defines a vector $\a$ of item sizes. The number of items is exactly $5|J|$, which can be seen for example by pairing each hyperedge item with 4 vertex items corresponding to that hyperedge.
Moreover, $|\a|=|\b|=|J|\cdot b_1$, as we  can sum up the items in $\a$ coordinate-wise as follows
\begin{align*}
\sum_{i,j} \left( [\eta_j,0,i \eta_j,3]+ (\mult_J(i,j)-1) [\eta_j,0,i\eta_j,\beta(j)]\right) + \sum_{ (w,x,y,z) \in J }[0^4,1,u-w,u-x,u-y,u-z,0] 
\\
=\sum_{i,j} \mult_J(i,j) [\eta_j,0^6] + \sum_{i,j} \mult_J(i,j) [0^5, i \eta_j + (u-i) \eta_j,0]  \\ +  |J|[0^4,1,0^5] + \sum_{i,j} [0^9, 3+(\mult_J(i,j)-1)\beta(j)]\\
= \sum_j |J| [\eta_j,0^6] + |J| [0^4,1,0^5] + |J|\sum_j [0^5,u\eta_j,0] + \sum_j [0^9,(3-\beta(j)) u] + \sum_j |J| [0^9,\beta(j)] = |J|b_1 , 
\end{align*}
where in the first equation we do the summation for each coordinate separately. Then we note that for every $j$ the total number of appearances of $(i,j)$ in an hyperedge is always $|J|$ since the matching is 4-partitite, i.e. $\sum_i \mult_j(i,j) = |J|$. We also note that $\sum_j (3-\beta(j) ) = 3*4 - 4*3=0$.


\begin{ex}\label{ex:items}
From the matchings in Figure~\ref{fig:join} we have
\begin{align*}
J=\{&(1,1,2,2),(2,2,1,1),(2,2,2,1),(3,3,3,3),(4,4,4,4),\\&(2,1,1,2),(2,1,2,3),(3,2,3,1),(4,3,1,3),(1,4,4,4)\}.
\end{align*}

The real vertex items we construct are
\begin{align*}
V=\{ &[1,0,0,0,0,i,0,0,0,3], [0,1,0,0,0,0,i,0,0,3],\\&[0,0,1,0,0,0,0,i,0,3],[0,0,0,1,0,0,0,0,i,3] \mid i=1,2,3,4\}
\end{align*}
and the dummy vertex items are of the form
$$ [1,0,0,0,0,i,0,0,0,4], [0,1,0,0,0,0,i,0,0,4],[0,0,1,0,0,0,0,i,0,4],[0,0,0,1,0,0,0,0,i,0]$$ 
for $i=1,2,3,4$,
but are repeated with the corresponding multiplicities given by $\mult_J(i,1)-1=1$ for $i=1,2,3,4$, also $\mult_J(1,2)-1=2$, $\mult_J(1,3)-1=2$, $\mult_J(1,4)-1=2$, $\mult_J(2,2)-1=2$, $\mult_J(2,3)-1=2$, $\mult_J(2,4)-1=1$, $\mult_J(3,2)-1=1$, $\mult_J(3,3)-1=1$, $\mult_J(3,4)-1=2$ and finally $\mult_J(4,j)-1=1$ for all $j$. The hyperedge items corresponding directly to the hyperedges listed in $J$ are as follows. From hyperedge $(1,1,2,2)$ we have item $[0,0,0,0,1,4-1,4-1,4-2,4-2,0]$, and similarly obtain the rest as
\begin{align*}
E=\{&[0^4,1,3,3,2,2,0],[0^4,1,2,2,3,3,0],[0^4,1,2,2,2,3,0],[0^4,1,1,1,1,1,0],[0^4,1,0,0,0,0,0],\\
&[0^4,1,2,3,3,2,0],[0^4,1,2,3,2,1,0],[0^4,1,1,2,1,3,0],[0^4,1,0,1,3,1,0],[0^4,1,3,0,0,0,0]\}
\end{align*}
We have $b_1=[1,1,1,1,1,4,4,4,4,12]$, $|J|=10$ and the bins are $\b = (b_1^{10})$.
\end{ex}

Up to this point, this was not different from \cite{GJ79}.
The real vertex item to $(i,j)=(u,4)$ (the vertex in the bottom right in Figure~\ref{fig:join}) is called the \emph{special vertex item}. By construction, it has size $[0,0,0,1,0,0,0,u,0,3]$ and is the unique item of this size.
The item to hyperedge $H=(u,u,u,u)$ is called the \emph{first special hyperedge item}.
The item to hyperedge $H'=(1,u,u,u)$ is called the \emph{second special hyperedge item}.
These two items are also the unique items with their respective sizes.

Note that the value of every coefficient is nonnegative and at most $\max\{4,u\}$.
Let
$$\delta  :=  |J|! \, \cdot \prod_{(i,j)\in[u]\times[4]}\big(\mult(i,j)-1\big)!
$$

\begin{lemma}
\label{lem:auxplacement}
In every $\vv K \in \cP(\a,\b)$, the special vertex item is put in a bin with either the first special hyperedge item or the second special hyperedge item, but not with both at the same time.
Let $\cP(\a,\b)_0$ be the subset of those $\vv K$ for which
the special vertex item is put in a bin with the first special hyperedge item,
and let $\cP(\a,\b)_1 = \cP(\a,\b)\setminus\cP(\a,\b)_0$.
Then
\[
\tfrac 1 \delta  \cdot  \big|\cP(\a,\b)_0\big|  =  \sharpIVDM\big(J\setminus\{H'\}\big) \quad\text{ and }\quad \tfrac 1 \delta
 \cdot  \big|\cP(\a,\b)_1\big|  =  \sharpIVDM\big(J\setminus\{H\}\big).
\]
\end{lemma}
\begin{proof}
Since $r$ is large and the size of the bins is $[1,1,1,1,1,u,u,u,u,12]$, a solution $\vv K$ to the instance must place exactly 5 items in every bin: One hyperedge item and four vertex items for some vertices $(i,j)\in[u]\times[4]$, one for each $j \in [4]$.
Moreover, since $r$ is large and since the 10th coefficient of the bin size is 12,
in a solution we must have that
in each bin the four vertex items are either all dummy vertex items or all real vertex items (because 12 can be written as $12=4+4+4+0=3+3+3+3$ if only the summands $0,3,4$ are available).

Now, since the \big(6th,{}7th,{}8th,{}9th\big) coordinates of a hyperedge item are $(u-w,u-x,u-y,u-z)$,
we conclude that each hyperedge $e$ must be placed together with vertex items for its four vertices from $\varphi(e)$ (real or dummy vertex items).
From a placement $\vv K$ like this we can create a solution to \sharpIVDM($J$) by selecting exactly those hyperedges that are in a bin with real vertex items.

In fact, there are $\delta$ many different placements $\vv K$ that result in the same 4D-matching: The bins can be permuted, and for each vertex the dummy vertex items can be permuted.
And vice versa: From a solution to \sharpIVDM($J$) we create $\delta$ many placements $\vv K$ of items by grouping the selected hyperedges together with their real vertex items, and grouping the unselected hyperedges together with their dummy vertex items.

These operations are inverses of each other, which gives a bijection between the set of 4D-matchings of $J$ and the set of cardinality $\delta$ subsets of $\cP(\a,\b)$ in which all elements arise from each other by permuting the bins and the dummy vertices.
Now \Cref{lem:joinminusH} implies the result.
\end{proof}

\begin{ex}
Continue with the example from Figure~\ref{fig:join} and Example~\ref{ex:items}. Since each bin has 5th coordinate equal to 1, and the only items with 5th coordinate equal to 1 are the hyperedge items, we must have one hyperedge item per bin. Looking at the first 4 coordinates all equal to 1, we must have exactly one vertex item for each $j=1,2,3,4$. As the last coordinate is 12, the only way to achieve it from the $3$s and $4$s is $12=3+3+3+3$, so 4 real vertices, or $12=4+4+4+0$, 4 dummy vertices. Hence, the items sum up to $b_1$ exactly when the 4 vertices are part of a hyperedge. So in a set partition instance we would have, for example, the following bin:
\begin{align*}[0,0,0,0,1,3,3,2,2,0]+& [1,0,0,0,0,1,0,0,0,3]+[0,1,0,0,0,0,1,0,0,3]\\
+&[0,0,1,0,0,0,0,2,0,3]+[0,0,0,1,0,0,0,0,2,3]=[1,1,1,1,1,4,4,4,4,12].
\end{align*}
\end{ex}

\subsection{The Modified \problem{SetPartition} Instance}
\label{subsec:modified}
We modify the item vector $\a$ from the construction above,
to obtain a vector $\c$ as follows.

\begin{itemize}
\item We add four items of sizes 1,2,4,5.
\item We increase the size of the special vertex item by 1. We decrease the size of the first special hyperedge item by 5.
We decrease the size of the second special hyperedge item by 2.
\end{itemize}
These seven items are the only items whose sizes are not divisible by~$r$.
W.l.o.g. let the special vertex item, the first special hyperedge item, and the second special hyperedge item be the first three item sizes in $\a$.
Then
$\c := (1,2,4,5,a_1+1,a_2-5,a_3-2,a_4,a_5,\ldots)$.
Let $\d := \b$.
We have $|\c|=|\d|+6$.
Finally, denote $\setpartition(J,H,H') := (\c,\d,\delta)$.
This completes the construction process we started in~$\S$\ref{subsec:auxsetpartition}.

\begin{lemma}
\label{lem:deltaPcd}
$\frac 1 \delta \rP(\c,(2,4,\d)) = \sharpIVDM(J\setminus\{H'\})$ \ and \
$\frac 1 \delta \rP(\c,(1,5,\d)) = \sharpIVDM(J\setminus\{H\})$.
\end{lemma}
\begin{proof}
The restrictions in the proof of \Cref{lem:auxplacement} still directly apply, because we only made small changes to the item sizes, which all were multiples of $r$, and $r$ is large.
The new items and the changed item sizes give exactly the additional constraints that the item sizes in each bin should add up to a number that is divisible by~$r$.

In $\rP(\c,(2,4,\d))$, the bin of size 2 must be filled with the item of size~2, and thus the bin of size~4 must be filled with the item left of size~4.
The special vertex item and the item of size 1 are placed with the second special hyperedge (because the item of size 2 has already been placed). The item of size~5 is placed with the first special hyperedge.
The remaining placements of items can be done as in the proof of \Cref{lem:auxplacement}.

In $\rP(\c,(1,5,\d))$, the bin of size 1 must be filled with the item of size~1, and the bin of size~5 must contain a small ($\leq 5$) odd item, but the only such item left is the item of size~5.
The parity now implies that the special vertex item is placed in a bin with the first special hyperedge item. The only two remaining small items of sizes 2 and~4 fill up the bins of the special hyperedge items.
The remaining placements of items can be done as in the proof of \Cref{lem:auxplacement}.
\end{proof}

\subsection{Putting the Pieces Together}

\begin{proof}[Proof of \Cref{thm:CeqPhardnessofchar} and \Cref{thm:PH-char}]
Recall that $\CeqP = [\GapP=0]$ and  $\PP = \mbox{$[\GapP\geq 0]$}$.
We prove both theorems simultaneously, so fix a comparison operator $\sim\ \in\{=,\geq\}$.

For every $L \in \CeqP$ there exist $F \in \sharpP$ and $F'\in\sharpP$ with $w\in L$ if and only if $F(w)\sim F'(w)$.
By the Cook--Levin theorem, there exists a polynomial-time algorithm
that on input $w$ outputs a Boolean circuit $C_w$ such that $F(w) = \problem{\#CircuitSAT}(C_w)$.
Analogously, there exists a polynomial-time algorithm
that on input $w$ outputs a Boolean circuit $C'_w$ such that $F'(w) = \problem{\#CircuitSAT}(C'_w)$.

Let $E := R(C_w)$ and $E':=R(C'_w)$, where $R$ is
the reduction from $\problem{\#CircuitSAT}$ to $\sharpIIIDM$,
defined as in~$\S$\ref{ss:3DM}.
Let $(J,H,H') := \join(E,E')$.
Let $(\c,\d,\delta) := \setpartition(J,H,H')$ as defined in~\S\ref{subsec:modified}.
Let $\la$ and $\alpha$ be from \Cref{p:char-part3}.
We have:
\begin{eqnarray*}
F(w) \sim F'(w)
&\Longleftrightarrow&
\problem{\#CircuitSAT}(C_w) \sim \problem{\#CircuitSAT}(C'_w)
\\
&\Longleftrightarrow&
\sharpIIIDM(E)\sim\sharpIIIDM(E')
\\
&\stackrel{\textup{Lem.~}\ref{lem:joinminusH}}{\Longleftrightarrow}&
\sharpIVDM(J\setminus\{H'\})\sim\sharpIVDM(J\setminus\{H\})
\\
&\stackrel{\textup{Lem.}~\ref{lem:deltaPcd}}{\Longleftrightarrow}&
\tfrac{1}{\delta}\rP(\c, (2,4,\d))\sim\tfrac{1}{\delta}\rP(\c, (1,5,\d))
\\
&\Longleftrightarrow&
\rP(\c, (2,4,\d)) \sim \rP(\c, (1,5,\d))
\\
&\stackrel{\textup{Pro.}~\ref{p:char-part3}}{\Longleftrightarrow}&
\chi^\la(\al) \sim 0.
\end{eqnarray*}
This completes the proof of both theorems.
\end{proof}

\begin{proof}[Proof of \Cref{thm:main}]
Combine \Cref{thm:CeqPhardnessofchar} and
\Cref{pro:collapse}.
\end{proof}

\medskip

\section{Additional Remarks}

\subsection{Combinatorial interpretations}\label{ss:finrem-Kron}
Finding positive combinatorial interpretations for the Kronecker, plethysm
and Schubert coefficients remains a central open problem in Algebraic Combinatorics.
Special cases for the Kronecker coefficients have been studied in
\cite{BO, BOR, Bla17, IMW,PP_strict,RW}, among many others.
Combinatorial interpretations for plethysm coefficients have been
even harder to find, see \cite{BBP,DIP,FI20} for some special cases.

For the Schubert coefficients, see \cite{Knu,KZ,Man,MPP}
for positive combinatorial interpretations in several special cases,
and \cite{ARY} for complexity of a related problem.
For the row character sums~$a_\la$ defined in~\eqref{eq:row-column-sums},
Frumkin \cite{Fru} proved that $a_\la\ge 1$ for all $|\la|>1$.
See also \cite{Sol} for a generalization to all finite groups.
We refer to \cite{KW} for a combinatorial interpretation of
an ingredient in the sum in~\eqref{eq:row-column-sums}, and to
\cite[p.~323]{Sun} for a connection to plethysm coefficients.
For the column character sums~$b_\la$, see \cite[Exc~7.69]{EC2}
and references therein.

\subsection{Unary vs binary input}
\label{ss:fin_rem_char_compl}
Our results are independent of the input encoding in the following sense:
the description size of $(\la,\pi)$ and $(\la,\mu)$ can differ exponentially
if $\mu$ is provided as a list of integers that are encoded in binary.
Our results hold in both of these settings.
It is noteworthy that such results do not exist for other quantities of
interest, for example the Kostka numbers, Littlewood--Richardson and
Kronecker coefficients, and the Schubert structure constants.

Narayanan \cite{Nar} proved that computing the \emph{Kostka coefficients} $K_{\la\mu}$ and
the LR-coefficients $c^{\la}_{\mu\nu}$ are $\sharpP$-complete when
the inputs $\la,\mu,\nu$ are encoded in binary.  It was conjectured in
\cite[Conj.\ 8.1]{PP_comp} that the LR-coefficients are $\sharpP$-complete
in unary.\footnote{The distinction between unary and binary input was underscored
in \cite{GJ78}.  Unfortunately, the original naming of ``strong''
vs.\ (the usual) ``weak'' $\NP$-completeness added to the confusion, and is best
to be avoided. }

We should note however, that the decision problems $[K_{\la\mu}=0]$ and $[c^{\la}_{\mu\nu}=0]$
are in $\P$ even when the input is binary. The first one reduces to checking the linear inequalities whether $\la \vartriangleright \mu$ in the dominance order. By the Knutson--Tao \emph{saturation theorem}~\cite{KT1}, the vanishing of
LR-coefficients reduces to checking if the \emph{Gelfand--Tsetlin polytope} is empty, see \cite{BI,MNS,DM06}.

The unary hardness of the counting problems would imply that the Schubert coefficients
are also $\sharpP$-hard to compute. Indeed, the natural encoding for the Schubert coefficients,
when the inputs are permutations, is in unary. On the other hand, the LR-coefficients are
special cases of the Schubert coefficients, but so far $\sharpP$-completeness is only known when $\la,\mu,\nu$ are encoded in binary.
Thus, we cannot yet conclude the computational hardness result.\footnote{This argument
points out the error in \cite[p.~885]{MQ} which concludes that Schubert coefficients
are $\sharpP$-hard.}
By contrast, the Kronecker coefficients of the symmetric group $g(\la,\mu,\nu)$ are
 $\sharpP$-hard with input in unary; this follows form the proof in~\cite{IMW} that
 vanishing of $g(\la,\mu,\nu)$ is $\NP$-hard in unary.

\subsection{GapP-completeness and parsimonious reductions}
To emphasize the difference, consider the following two problems:

\medskip

\noindent\begin{minipage}{\textwidth}
Problem  $\problem{ComputeCharUnary}$:
\begin{compactitem}
 \item Input: An integer~$n$, and partitions $\la,\mu\vdash n$, as lists of numbers encoded in unary
 \item Output: $\chi^\la(\mu)$
\end{compactitem}
\end{minipage}

\medskip

\noindent\begin{minipage}{\textwidth}
Problem  $\problem{ComputeCharBinary}$:
\begin{compactitem}
 \item Input: An integer~$n$, and partitions $\la,\mu\vdash n$, as lists of numbers encoded in binary
 \item Output: $\chi^\la(\mu)$
\end{compactitem}
\end{minipage}

\medskip

\noindent
As we mentioned in the introduction, Hepler \cite{Hep} proved that computing $\chi^\la(\mu)$ is
$\sharpP$-hard in unary, and thus in binary.\footnote{In \cite{PP_comp}, the second and third authors made
erroneous claims on this point.}
The following result has not been observed before,
but follows directly from \Cref{p:char-part}:

\begin{theorem}
\label{thm:strongGapPhard}
The problem \problem{ComputeCharBinary} is $\GapP$-complete under Turing reductions.
\end{theorem}

We note that we cannot at this point strengthen the result to parsimonious many-one reductions, because the reduction from matchings to counting ordered set partitions is itself not parsimonious, having the factor of~$\delta$.

\begin{conjecture}
\label{conj:strongGapPhard}
The problem \problem{ComputeCharBinary} is $\GapP$-complete under many-one reductions.
\end{conjecture}

We should note though that the reduction from \problem{\#SetPartition} to \problem{ComputeCharUnary}  is parsimonious from the following, see also~\cite{Hep} for a closely related reduction:

\begin{proposition}
Let $\a$ and $\b$ be two positive sequences with equal sums, and let $\b$ have $p-1$ many parts. Let $\la=p\,\sort(\b)$ and $\al = p\a$. Then
$$\chi^\la(\al) = \rP(\a,\b).$$
\end{proposition}

The proof follows directly from applying Proposition~\ref{p:char-part} and observing that since all sizes~$\al$ are divisible by~$p$, we must have that $p$ divides $(\la_i + \sigma_i-i) = (pb_i+\sigma_i-i)$ for all bin sizes. Then $\sigma_i =i$, and the only nonzero term which survives is $\rP(\al,\la) = \rP(\a,\b)$.

\subsection{Combinatorial identities}\label{ss:finrem-White}
The irreducible characters of a finite group $G$
are orthonormal with respect to the inner product
$$\langle \chi, \psi \rangle = \frac{1}{|G|} \sum_{g \in G}  \chi(g)  \psi(g^{-1}),
$$
see e.g.~\cite[$\S$2.3]{Serre}.  Thus, equation~\eqref{eq:charsum} gives
the squared norm of a character $\chi^\la$.

Equation \eqref{eq:RSK} for general finite groups is called \emph{Burnside's identity}
\cite[$\S208$]{Burn} and can be generalized as follows.
For every  partition $\mu=(1^{m_1} \ldots\ell^{m_\ell})\vdash n$ with $m_i$ parts of size~$i$, we have:
\begin{equation}
\label{eq:charsumz}
1^{m_1} m_1! \,  \,  \cdots \, \ell^{m_\ell}\ell! \, = \ \sum_{\la\vdash n} \big(\chi^\la(\mu)\big)^2,
\end{equation}
see e.g.\ \cite[Thm~1.10.3]{Sag}.  When $\mu=(1^n)$ we get~\eqref{eq:RSK}, but in this case finding a natural combinatorial partition of the objects from the LHS to sets of sizes given by the character squares
is unlikely for the same reason as for \eqref{eq:charsum}.

It would be interesting to see if \eqref{eq:charsumz} has a combinatorial interpretation
for \emph{some} classes of~$\mu$.  For example, when $\mu=(n)$, the characters
$\chi^\la(\mu) \in \{0,\pm 1\}$ and there is an easy combinatorial interpretation for the character squares $\big(\chi^\la(\mu)\big)^2$.  More generally, for $\mu=(k^{n/k})$,
all rim hook tableaux in the Murnaghan--Nakayama rule for $\chi^\la(\mu)$
have the same sign, see e.g.\ \cite[$\S$2.7]{JK} and \cite{SW},
so again character squares have a combinatorial interpretation.
These ``equal cycles'' characters also appear in the mysterious identities
in \cite[Thm~3.3]{KK}.  We note that they do not have a combinatorial
proof except for the first identity which coincides with~\eqref{eq:charsumz}.

\subsection{Other values}\label{ss:finrem-other}
As discussed e.g.\ in \cite[$\S$8]{PPV} and \cite{Mil,Pel},
 other values of the character table are of interest as well,
notably the uniqueness and parity of the characters.  The corresponding
complexity problems $\big[\chi^\la(\mu)=1\big]$ and
$\big[\chi^\la(\mu)=0~\text{mod}~2\big]$ are also very interesting
and worth studying.

\subsection*{Acknowledgements}
We are grateful to Swee Hong Chan,
Josh Grochow, Sam Hopkins, Greg Kuperberg and
Alejandro Morales for fruitful discussions.
We thank an anonymous reviewer for helping to improve this paper's presentation.
The first author was partially supported by the DFG grant IK 116/2-1 and the EPSRC grant EP/W014882/1.
The second author was partially supported by the NSF grant CCF-2007891.
The third author was partially supported by the NSF grant CCF-2007652.
For the purpose of open access, the authors have applied a Creative Commons Attribution (CC BY) license to any Author Accepted Manuscript version arising from this submission.

\end{document}